\documentclass[12pt]{amsart}

\usepackage{amsmath,amssymb,amsfonts,amsthm,latexsym,graphicx,multirow}
\usepackage{hyperref}

\def\A{\mathrm{A}} \def\AGL{\mathrm{AGL}}   \def\Aut{\mathrm{Aut}}
 
 \def\calB{\mathcal{B}}     \def\Cos{\mathsf{Cos}}
\def\D{\mathrm{D}}

 \def\Hol{\mathrm{Hol}}

\def\M{\mathrm{M}} 

\def\Nor{\mathbf{N}}

  \def\PGL{\mathrm{PGL}}    \def\PSL{\mathrm{PSL}}  \def\PSiU{\mathrm{P\Sigma U}}  \def\PSp{\mathrm{PSp}} \def\PSU{\mathrm{PSU}}
\def\Q{\mathrm{Q}}
  \def\Ree{\mathrm{Ree}} 
  \def\Soc{\mathrm{Soc}}     \def\Sy{\mathrm{S}} \def\Sym{\mathrm{Sym}} \def\Sz{\mathrm{Sz}}
\def\Z{\mathrm{C}} 

\newtheorem{theorem}{Theorem}[section]
\newtheorem{lemma}[theorem]{Lemma}

\theoremstyle{definition}
\newtheorem{definition}[theorem]{Definition}

\newtheorem{example}[theorem]{Example}

\newtheorem{problem}[theorem]{Problem}

\begin{document}

\title{On Isomorphisms of Vertex-transitive Graphs}

\thanks{This work was supported by NSFC grant (11501011, 11501188) and Aid Program for Science and Technology Innovative Research Team in Higher Educational Institutions of Hunan Province.}

\author[Chen]{Jing Chen}
\address{(Chen) School of Mathematics\\
Hunan First Normal University\\
Changsha 410205\\
and Center for Discrete Mathematics and Theoretical Computer Science\\
Fuzhou University\\
Fuzhou 350003\\
P. R. China}
\email{chenjing827@126.com}

\author[Xia]{Binzhou Xia}
\address{(Xia) Beijing International Center for Mathematical Research\\
Peking University\\
Beijing, 100871\\
P. R. China}
\email{binzhouxia@pku.edu.cn}


\maketitle

\begin{abstract}
The isomorphism problem of Cayley graphs has been well studied in the literature, such as characterizations of CI (DCI)-graphs and CI (DCI)-groups. In this paper, we generalize these to vertex-transitive graphs and establish parallel results. Some interesting vertex-transitive graphs are given, including a first example of connected symmetric non-Cayley non-GI-graph. Also, we initiate the study for GI and DGI-groups, defined analogously to the concept of CI and DCI-groups.
\end{abstract}


\section{Introduction}

Throughout this paper, by (di)graph we mean finite digraph without loops or multiedges, and all groups are assumed to be finite. Deciding whether two graphs are isomorphic is fundamental for the study of graphs, especially for determining isomorphism classes of graphs.
A graph is said to be $G$-vertex-transitive if the subgroup $G$ of its full automorphism group acts transitively on the vertex set.
One would expect to determine the isomorphisms between two $G$-vertex-transitive graphs by the information of the group $G$.
For Cayley graphs, such an approach was initiated by a conjecture of \'{A}d\'{a}m in 1967 \cite{Adam}, and has been extensively studied over the past decades,
see for example \cite{Alspach,Babai,Dobson-95,Godsil,Kovacs-M,Muzychuck,Palfy,Somlai,Spiga} and more references listed in the survey \cite{CI-survey}.
Since a large number of vertex-transitive graphs are not Cayley graphs, it is natural to extend the study from Cayley graphs to vertex-transitive graphs.
The isomorphism problem for metacirculants (not necessarily Cayley graphs) has been considered by Dobson \cite{Dobson}.

To be precise, we need the concept of coset graphs.
Let $\Gamma=(V,E)$ be a $G$-vertex-transitive graph, $\alpha$ be a vertex of $\Gamma$ and $S$ be the set of elements of $G$ which maps $\alpha$ to its (out) neighbors.
Then $\Gamma$ is uniquely determined by the triple $(G,G_\alpha,S)$ in the following sense:
writing $H=G_\alpha$ and identifying the vertex set $V$ with the set $[G{:}H]$ of right cosets of $H$ in $G$,
the action of $G$ on $V$ is equivalent to the action of $G$ on $[G{:}H]$ by right multiplication.
In particular, if $\alpha$ is identified with $H\in[G{:}H]$ then
the neighborhood $\Gamma(\alpha)$ consists of $Hg$ with $g\in S$, and moreover, $Hx\sim Hy$ if and only if $yx^{-1}\in HSH$.
This defines a \emph{coset graph} representation of $\Gamma$, denoted by $\Cos(G,H,HSH)$.
Note that $H$ is core-free in $G$ (that is, $H$ does not contain any nontrivial normal subgroup of $G$) since $G$ is a transitive permutation group on $V$, and $S\subseteq G\setminus H$ since $\Gamma$ has no loops.
Clearly, for any automorphism $\tau\in\Aut(G)$ we have $\Cos(G,H,HSH)\cong\Cos(G,H^\tau,H^\tau S^\tau H^\tau)$.

\begin{definition}
The $G$-vertex-transitive graph $\Gamma=\Cos(G,H,HSH)$ is called a \emph{GI-graph} (`GI' stands for `Group automorphism inducing Isomorphism')
of $G$ if for any graph $\Sigma=\Cos(G,H,HTH)$ with $T\subseteq G\setminus H$ and $\Gamma\cong\Sigma$, there exists
$\tau\in\Aut(G)$ such that $H^\tau=H$ and $HS^\tau H=HTH$.
A group $G$ is called a \emph{DGI-group} (`D' emphasizes that our graph may be Directed) if each $G$-vertex-transitive graph is a GI-graph of $G$.
A group $G$ is called a \emph{GI-group} if each undirected $G$-vertex-transitive graph is a GI-graph of $G$.
\end{definition}

Note that $\Cos(G,1,S)$ is a Cayley graph of $G$, and the GI-graphs of $G$ with $H=1$ are exactly the so called \emph{CI-graphs} of $G$. If each Cayley graph of $G$ is a CI-graph of $G$, then $G$ is called a \emph{DCI-group}. If each undirected Cayley graph of $G$ is a CI-graph of $G$, then $G$ is called a \emph{CI-group}. Clearly, a DGI-group is necessarily a GI-group, and a DGI-group (GI-group) is necessarily a DCI-group (CI-group). A small list of candidates for DCI and CI-groups has been obtained, through the effort of many mathematicians, see~\cite[Theorem~8.7]{CI-survey} and \cite[Corollary~1.5]{LPX1}. However, determining which groups in the list are indeed DCI or CI-groups is not easy and largely open. As being DGI-groups (GI-groups) is more restrictive than being DCI-groups (CI-groups), the explicit list of DGI-groups (GI-groups) would be smaller than that of DCI-groups (CI-groups). Thus we propose the problem:

\begin{problem}\label{prob1}
Classify the finite DGI-groups (GI-groups).
\end{problem}

In the literature, a crucial step to solve a conjecture of Babai and Frankel \cite{Babai-F}
stating that CI-groups are solvable was to determine whether there exists a non-CI-Cayley graph of $\A_5$.
After 20 years since Babai-Frankel conjecture was posed, a non-CI-Cayley graph of $\A_5$ of valency $29$ was constructed by Li \cite{CI-soluble}, thus completing the proof of the conjecture. Although some other non-CI-Cayley graphs of $\A_5$ was later constructed in \cite{Conder,XX}, Li's graph is the only known connected symmetric non-CI-graph of $\A_5$ yet. Here
a graph $\Gamma$ is called \emph{$G$-symmetric} for some $G\leqslant\Aut(\Gamma)$ if $G$ acts transitively on the arc set of $\Gamma$, and $\Gamma$ is simply called \emph{symmetric} if $\Gamma$ is $\Aut(\Gamma)$-symmetric. In general, constructing connected symmetric non-GI-graphs is not easy. Due to the significance of non-CI-Cayley graphs of $\A_5$, one would ask:

\begin{problem}\label{prob2}
Does there exist a connected symmetric non-GI-graph of $\A_5$ other than Li's?
\end{problem}

The layout of this paper is as follows. After this introduction, we give the criterion for GI-graph in Section~\ref{sec5}, which enables us to construct GI and non-GI-graphs, respectively, in Section~\ref{sec2}. In particular, we prove the theorem below by Example~\ref{exam1}.

\begin{theorem}\label{thm3}
There exists a connected symmetric non-Cayley non-GI-graph of order $40$ and valency $12$.
\end{theorem}

\noindent Then in Section~\ref{sec3} we establish some results on Problem~\ref{prob1}. The final section is devoted to Problem~\ref{prob2}, where it is shown that a connected $\A_5$-symmetric graph is necessarily GI if its full automorphism group is almost simple or vertex-primitive.

\section{Criterion for GI-graph}\label{sec5}

As mentioned in the introduction, $G$-vertex-transitive graphs can be represented
as coset graphs of $G$: for a core-free subgroup $H$ of $G$ and a subset $S\subseteq G\setminus H$, define
$\Gamma=\Cos(G,H,HSH)$ to be the graph with vertex set $V:=[G{:}H]$ such that $Hx\sim Hy$ if and only if $yx^{-1}\in HSH$.
For any $g\in G$, the right multiplication of $g$ on the cosets in $[G{:}H]$ gives an element of $\Sym(V)$, denoted by $\hat{g}$.
Moreover, denote $\hat{G}=\{\hat{g}\mid g\in G\}$.
(The reader should be aware that this also depends on the subgroup $H$ although the $\hat{\ }$ symbol does not indicate.)
We list here some basic facts concerning coset graphs.

\begin{lemma}\label{lem1}
Let $\Gamma=\Cos(G,H,HSH)$.
\begin{itemize}
\item[(a)] $\Gamma$ is undirected if and only if $HSH=HS^{-1}H$, where $S^{-1}:=\{s^{-1}\mid s\in S\}$.
\item[(b)] $G$ acts faithfully and transitively on the vertex set $[G{:}H]$ by right multiplication, so $\hat{G}$ is a subgroup of $\Aut(\Gamma)$ isomorphic to $G$.
\item[(c)] $\Gamma$ is connected if and only if $\langle H,S\rangle=G$.
\item[(d)] $\Gamma$ is $G$-symmetric if and only if $HSH=HgH$ for some $g\in G$. In this case, the valency of $\Gamma$ is equal to $|H|/|H^g\cap H|$.
\end{itemize}
\end{lemma}

Let $X$ and $Y$ be permutation groups on $\Omega$ and $\Delta$, respectively. We say that $X$ is \emph{permutation isomorphic} to $Y$ if there exist a bijection $\sigma:\Omega\rightarrow\Delta$ and a group isomorphism $\varphi:X\rightarrow Y$ such that $(\alpha^x)^\sigma=(\alpha^\sigma)^{\varphi(x)}$ for any $\alpha\in\Omega$ and $x\in X$. The following folklore theorem is an extension of the criterion for a Cayley graph to be a CI-graph \cite{Alspach-P,Babai} to those vertex-transitive graphs.
The proof goes along the same lines as that of the CI-graph criterion, so we omit it.

\begin{theorem}\label{thm2}
A $G$-vertex-transitive graph $\Gamma$ is a GI-graph of $G$ if and only if subgroups of $\Aut(\Gamma)$
which are permutation isomorphic to $\hat{G}$ are all conjugate in $\Aut(\Gamma)$.
\end{theorem}

Based on Theorem~\ref{thm2}, we establish a sufficient condition on GI-graphs as follows.

\begin{theorem}\label{thm1}
Suppose that $G$ is a finite group of odd order, $p$ is the smallest prime divisor of $|G|$,
$\Gamma$ is a $G$-vertex-transitive graph and $A$ is the full automorphism group of $\Gamma$.
For any vertex $\alpha$ of $\Gamma$, if $\gcd(|G|,|A_\alpha|)=1$, then $\Gamma$ is a GI-graph of $G$.
In particular, if $\Gamma$ is connected of valency less than $p$, then $\Gamma$ is a GI-graph of $G$.
\end{theorem}

\begin{proof}
Since $G$ is transitive on the vertices of $\Gamma$ we have $A=GA_\alpha$.
Assume that $\gcd(|G|,|A_\alpha|)=1$.
Then $G$ is a Hall $\pi$-subgroup of $A$, where $\pi$ is the set of the prime divisors of $|G|$.
Note that $\pi$ is a set of odd primes as $|G|$ is odd.
Then for any $\sigma\in\Sym(V)$ with $G^\sigma\leqslant A$, one deduces from \cite[Theorem~A]{Gross}
that $G$ and $G^\sigma$ are conjugate in $A$ as they are Hall $\pi$-subgroups of $A$.
Hence according to Theorem~\ref{thm2}, $\Gamma$ is a GI-graph of $G$.

Now assume that $\Gamma$ is connected of valency less than $p$.
It suffices to prove that $\gcd(|G|,|A_\alpha|)=1$.
Suppose for a contradiction that there exists a prime number $r$ dividing $\gcd(|G|,|A_\alpha|)$
and that $R$ is a Sylow $r$-subgroup of $A_\alpha$.
Since $\Gamma$ is connected, there exist a neighbor $\beta$ of $\alpha$ and an element $x\in R$
such that $\beta^x\neq\beta$.
It follows that the orbit of $\beta$ under $\langle x\rangle$ has length at least $r$,
contrary to our assumption that the valency of $\Gamma$ is less than $p\leqslant r$.
\end{proof}

Below is a necessary condition for GI-graphs.

\begin{theorem}\label{Component}
If $\Cos(G,H,HSH)$ is a GI-graph of a group $G$, then for any embedding $\varphi:\langle H,S\rangle\rightarrow G$ such that $H^\varphi=H$,
there exists $\tau\in\Aut(G)$ such that $H^\tau=H$ and $\langle H,S\rangle^\tau=\langle H,S^\varphi\rangle$.
\end{theorem}

\begin{proof}
Note that $\Cos(\langle H,S\rangle,H,HSH)$ and $\Cos(\langle H,S^\varphi\rangle,H,HS^\varphi H)$
are connected components of $\Cos(G,H,HSH)$ and $\Cos(G,H,HS^\varphi H)$, respectively. Then
$$
\Cos(G,H,HSH)\cong\Cos(G,H,HS^\varphi H)
$$
if and only if $\Cos(\langle H,S\rangle,H,HSH)\cong\Cos(\langle H,S^\varphi\rangle,H,HS^\varphi H)$.
As $\varphi$ induces an graph isomorphism from $\Cos(\langle H,S\rangle,H,HSH)$ to
$\Cos(\langle H,S^\varphi\rangle,H,HS^\varphi H)$, we thus have an isomorphism $\Cos(G,H,HSH)\cong\Cos(G,H,HS^\varphi H)$.
Since $\Cos(G,H,HSH)$ is a GI-graph of $G$, there exists $\tau\in\Aut(G)$ such that $H^\tau=H$ and $HS^\tau H=HS^\varphi H$.
Consequently,
$$
\langle H,S\rangle^\tau=\langle H,HSH\rangle^\tau=\langle H,HS^\tau H\rangle=\langle H,HS^\varphi H\rangle=\langle H,S^\varphi\rangle,
$$
which completes the proof.
\end{proof}

\section{Examples}\label{sec2}

First of all, the complete graphs and their complements are GI-graphs.
We regard them as \emph{trivial} GI-graphs.
An observation of \cite{LPX} says that every finite group of order greater than two has non-trivial CI-graphs.
Thus we know that every finite group of order greater than two has non-trivial GI-graphs.
Given a finite group $G$ of odd order, recall that as Theorem~\ref{thm1} asserts, every $G$-vertex-transitive graph $\Gamma$ of valency less than the smallest prime divisor of $|G|$ is a GI-graph of $G$.
This provides us with more examples of GI-graphs.

A \emph{$2$-arc} of a graph $\Gamma$ is a triple $(\alpha,\beta,\gamma)$ of pairwise distinct vertices of $\Gamma$ such that $\alpha\sim\beta$ and $\beta\sim\gamma$. A graph is said to be \emph{$(G,2)$-arc-transitive} for some $G\leqslant\Aut(\Gamma)$ if $G$ acts transitively on the set of $2$-arcs. Recall that the \emph{socle} of a group $G$ is the product of all its minimal normal subgroups, denoted by $\Soc(G)$. We call a group \emph{almost simple} if its socle is nonabelian simple. It is readily seen that the almost simple groups with a given socle $T$ are precisely those groups $G$ satisfying $T\leqslant G\leqslant\Aut(T)$, whence $G/T$ is a solvable group by the well-known Schreier conjecture. The next example follows from \cite[Theorem~1.3]{Fang} and the criteria in Theorem~\ref{thm2}.

\begin{example}
Let $G$ be an almost simple group with socle $\Sz(2^{2n+1})$ or $G=\Ree(3^{2n+1})$.
Then every connected undirected $(G,2)$-arc transitive graph is a GI-graph of $G$.
\end{example}

Utilizing Theorem~\ref{Component}, we are able to construct some disconnected non-GI-graphs.

\begin{example}
Let $m$ and $n$ be integers such that $m\geqslant2$ and $n\geqslant2m+6$. Take $G=\A_n$, $a=(5,6)(7,8,\dots,2m+5,2m+6)\in G$, $b=(1,2)(3,4)a\in G$, $H=\langle a^2\rangle=\langle b^2\rangle$, $S=\{a,a^3,\dots,a^{2m-3},a^{2m-1}\}$ and $\varphi:a^i\mapsto b^i$ for any $i\in\mathbb{Z}$. Then $\varphi$ is an embedding of $\langle H,S\rangle$ into $G$ such that $H^\varphi=H$ and $\langle H,S^\varphi\rangle=\langle b\rangle$. Apparently, there does not exist $\tau\in\Aut(G)$ such that $\langle H,S\rangle^\tau=\langle a\rangle^\tau=\langle b\rangle$. Hence by Theorem~\ref{Component}, the coset graph $\Cos(G,H,HSH)$ is non-GI.
\end{example}

We close this section with the construction of a connected symmetric non-Cayley non-GI-graph, which proves Theorem~\ref{thm3}.

\begin{example}\label{exam1}
Let $X=\PSL_4(3)$ acting naturally on the set $\Omega$ of one-dimensional subspaces of $\mathbb{F}_3^4$, a four-dimensional vector space over $\mathbb{F}_3$. Take $\alpha\in\Omega$, and $G=\PSiU_4(2)=\PSp_4(3){:}\Z_2$ to be a maximal subgroup of $X$. There exists an involution $g\in G$ such that $\langle G_\alpha,g\rangle=G$ and $|G_\alpha|/|G_\alpha^g\cap G_\alpha|=12$. Let $\Gamma=\Cos(G,G_\alpha,G_\alpha gG_\alpha)$. Then $\Gamma$ is a connected $G$-symmetric and $G$-vertex-primitive graph of order $|\Omega|=40$ and valency $12$. Moreover, $G$ has two conjugacy classes of subgroups isomorphic to $\Sy_6$, fused in $X$, and the groups in both conjugacy classes are transitive on $\Omega$. Take $P$ to be a group in one of these two conjugacy classes, and $Q$ be a group in the other. Since $P$ and $Q$ are conjugate in $X$, they are permutation isomorphic. We claim that $\Gamma$ is a non-Cayley non-GI-graph of $P$.

In fact, the conclusion that $\Gamma$ is not a Cayley graph is obvious as $|P|\neq40$. Denote $Y=\Aut(\Gamma)$. In light of Theorem~\ref{thm2}, it suffices to show that $P$ and $Q$ are not conjugate in $Y$. If $\Soc(Y)=\Soc(G)=\PSU_4(2)$, then $Y=G$ since $G\leqslant Y$ and $G=\Aut(\PSU_4(2))$, which indicates that $P$ and $Q$ are not conjugate in $Y$, as desired. Assume next that $\Soc(G)\neq\Soc(Y)$. Then there exists a subgroup $H$ of $Y$ such that $\Soc(G)\neq\Soc(H)$ and $G$ is maximal in $H$. By~\cite{LPS1987}, either $G$ is maximal in $\A_{40}G$, or $H$ is almost simple with socle $\PSL_4(3)$. For the latter, $H_\alpha=\Z_3^3{:}\PSL_3(3)$ or $\Z_3^3{:}(\PSL_3(3)\times\Z_2)$ since $H$ is primitive on $40$ points, but then $H_\alpha$ does not have a subgroup of index $12$, violating the requirement that $\Gamma$ is $H$-symmetric as $\Gamma$ is $G$-symmetric. Therefore, $G$ is maximal in $\A_{40}G$, and hence $Y\cap\A_{40}G=G$ or $\A_{40}G$. Because $\Gamma$ is not a complete graph, we have $Y\ngeqslant\A_{40}$. It follows that $Y\cap\A_{40}G=G$. If $G\nleqslant\A_{40}$, then $\A_{40}G=\Sy_{40}$ and thus $Y=Y\cap\Sy_{40}=G$, contrary to our assumption that $\Soc(G)\neq\Soc(Y)$. Consequently, $G\leqslant\A_{40}$, and so $G$ has index two in $Y$. Since $\Soc(G)$ is a minimal normal subgroup of $Y$ and $\Soc(G)\neq\Soc(Y)$, we conclude that $Y$ has a minimal normal subgroup other than $\Soc(G)$, say $N$. Viewing that $N\nleqslant G$, we have $N=\Z_2$ and $Y=G\times N$. Hence $P$ and $Q$ are not conjugate in $Y$, proving our claim.
\end{example}

\section{GI-groups}\label{sec3}

A group $G$ is said to be \emph{Hamiltonian} if every subgroup of $G$ is normal. It is obvious that abelian groups are all Hamiltonian, but the converse is not true (for instance, the quaternion group $\Q_8$ is Hamiltonian but not abelian).

\begin{lemma}\label{lem4}
Let $G$ be a Hamiltonian group. Then $G$ is DGI (GI) if and only if $G$ is DCI (CI).
\end{lemma}

\begin{proof}
For any coset graph $\Cos(G,H,HSH)$ of $G$, the condition that $H$ is core-free in $G$ forces $H=1$ since $G$ is Hamiltonian. This means that each coset graph of $G$ is a Cayley graph of $G$. Hence the concepts of DGI (GI) and DCI (CI) coincide.
\end{proof}

Lemma~\ref{lem4} immediately shows up some DGI-groups (GI-groups) from the list of DCI-groups (CI-groups). For example, since the groups $\Z_k$, $\Z_{2k}$ and $\Z_{4k}$, where $k$ is odd square-free, are Hamiltonian and DCI \cite{Muzychuck95,Muzychuck97} simultaneously, we know that they are DGI-groups.

\begin{theorem}
$\D_{2p}$ is a DGI-group for any odd prime $p$.
\end{theorem}

\begin{proof}
Let $G=\D_{2p}$, $N$ be the Sylow $p$-subgroup of $G$, and $\Gamma=\Cos(G,H,HSH)$ be a coset graph of $G$ with vertex set $V=[G{:}H]$, where $H$ is a core-free subgroup of $G$ and $S\subseteq G\setminus H$. If $H=1$, then $\Gamma$ is a DGI-graph of $G$ by~\cite{Babai}. Hence we assume that $H\neq1$. As $H$ is core-free in $G$, we conclude that $H=\Z_2$ and $|V|=|G|/|H|=p$. Let $X$ be a subgroup of $\Aut(\Gamma)$ such that $X=\varphi^{-1}\hat{G}\varphi$ for some $\varphi\in\Sym(V)$, and $Y$ be a Sylow $p$-subgroup of $X$. Then $Y=\Z_p$, and by the Sylow theorem, there exists $\tau\in\Aut(\Gamma)$ such that $Y=\tau^{-1}\hat{N}\tau$. It derives from $X=\varphi^{-1}\hat{G}\varphi$ that $Y=\varphi^{-1}\hat{N}\varphi$. Thereby we obtain $\varphi^{-1}\hat{N}\varphi=\tau^{-1}\hat{N}\tau$, or equivalently, $\varphi\tau^{-1}\in\Nor_{\Sym(V)}(\hat{N})$. Note that $\Nor_{\Sym(V)}(\hat{N})\leqslant\Nor_{\Sym(V)}(\hat{G})$. This leads to $\varphi\tau^{-1}\in\Nor_{\Sym(V)}(\hat{G})$ and thus
$$
X=\varphi^{-1}\hat{G}\varphi=\tau^{-1}(\varphi\tau^{-1})^{-1}\hat{G}(\varphi\tau^{-1})\tau=\tau^{-1}\hat{G}\tau.
$$
Now appealing Theorem~\ref{thm2} we know that $\Gamma$ is a DGI-graph of $G$, which proves the lemma.
\end{proof}

We close this section with a theorem stating that being DGI-groups (GI-groups) is inherited by subgroups.

\begin{theorem}
If $G$ is a DGI-group (GI-group), then any subgroup $H$ of $G$ is a DGI-group (GI-group).
\end{theorem}

\proof
Suppose that $G$ is a DGI-group (GI-group). Let $\Gamma=\Cos(H,K,KSK)$ and $\Sigma=\Cos(H,K,KTK)$ be two isomorphic (undirected) coset graphs of $H$, where $K$ is a core-free subgroup of $H$ and $S,T$ are subsets of $H\setminus K$. Clearly, $K$ is also core-free in $G$. Without loss of generality we assume that $S$ and $T$ are both unions of double cosets of $K$.

First assume that $\langle K,S\rangle=H$. Then $\Gamma$ is connected, and so is $\Sigma$ since $\Gamma\cong\Sigma$. Noticing that $\Cos(G,K,KSK)$ and $\Cos(G,K,KTK)$ are $|G|/|H|$ copies of $\Gamma$ and $\Sigma$, respectively, we have $\Cos(G,K,KSK)\cong\Cos(G,K,KTK)$. Then as $G$ is a DGI-group (GI-group), there exists $\tau\in\Aut(G)$ such that $K^\tau=K$ and $KS^\tau K=KTK$. It follows that
$$
H^\tau=\langle K,S\rangle^\tau=\langle K,KSK\rangle^\tau=\langle K^\tau,KS^\tau K\rangle=\langle K,KTK\rangle=\langle K,T\rangle=H.
$$
This shows that $\tau$ induces an automorphism of $H$.

Next assume that $\langle K,S\rangle\neq H$. Then $|K\cup S|\leqslant|H|/2$, and so
$$
|K\cup(H\setminus S)|=|K|+|H\setminus(K\cup S)|>|H\setminus(K\cup S)|\geqslant|H|/2.
$$
Let $\overline{S}=(H\setminus S)\setminus K$ and $\overline{T}=(H\setminus T)\setminus K$. Then $\langle K,\overline{S}\rangle=\langle K,H\setminus S\rangle=H$, which means that the complement graph $\overline\Gamma$ of $\Gamma$ is connected and so is the complement graph $\overline\Sigma$ of $\Sigma$. From $\Gamma\cong\Sigma$ we deduce $\Cos(H,K,K\overline{S}K)=\overline\Gamma\cong\overline\Sigma=\Cos(H,K,K\overline{T}K)$. Hence $\Cos(G,K,K\overline{S}K)\cong\Cos(G,K,K\overline{T}K)$, and there exists $\tau\in\Aut(G)$ such that $K^\tau=K$ and $K\overline{S}^\tau K=K\overline{T}K$ since $G$ is a DGI-group (GI-group). As a consequence,
$$
H^\tau=\langle K,\overline{S}\rangle^\tau=\langle K,K\overline{S}K\rangle^\tau=\langle K^\tau,K\overline{S}^\tau K\rangle=\langle K,K\overline{T}K\rangle=\langle K,\overline{T}\rangle=H,
$$
showing that $\tau$ induces an automorphism of $H$. Moreover,
$$
KS^\tau K=(H\setminus K)\setminus(K\overline{S}^\tau K)=(H\setminus K)\setminus(K\overline{T}K)=KTK.
$$

Thereby we conclude that there always exists $\tau\in\Aut(G)$ such that $K^\tau=K$ and $KS^\tau K=KTK$. This implies that $H$ is a DGI-group (GI-group).
\qed

%

\section{GI-properties of connected $\A_5$-symmetric graphs}\label{sec4}

For a group $G$, the expression $G=HK$ with proper subgroups $H$ and $K$ of $G$ is called a \emph{factorization} of $G$. The lemma below can be read off from~\cite{Xia}.

\begin{lemma}\label{lem2}
If $T=GK$ is a factorization of a simple group $T$ with $G=\A_5$,
then either $(T,K)=(\A_n,\A_{n-1})$ with $n\in\{10,12,15,20,30,60\}$ or $(T,K)$ lies in \emph{Table~\ref{tab1}}.
\end{lemma}

\begin{table}[htbp]
\caption{}\label{tab1}
\centering
\begin{tabular}{|l|l|l|l|}
\hline
row & $T$ & $K$ \\
\hline
1 & $\A_6$ & $\A_4$, $\Sy_4$, $\Z_3^2{:}\Z_4$, $\A_5$ \\
2 & $\A_7$ & $\PSL_2(7)$ \\
3 & $\A_8$ & $\AGL_3(2)$ \\
4 & $\PSL_2(11)$ & $\Z_{11}$, $\Z_{11}{:}\Z_5$ \\
5 & $\PSL_2(19)$ & $\Z_{19}{:}\Z_9$ \\
6 & $\PSL_2(29)$ & $\Z_{29}{:}\Z_7$, $\Z_{29}{:}\Z_{14}$ \\
7 & $\PSL_2(59)$ & $\Z_{59}{:}\Z_{29}$ \\
8 & $\M_{12}$ & $\M_{11}$ \\
\hline
\end{tabular}
\end{table}

The following two theorems are the main results of this section.

\begin{theorem}\label{lem3}
Let $G=\A_5$ and $\Gamma$ be a connected symmetric coset graph of $G$. If $\Aut(\Gamma)$ is almost simple, then $\Gamma$ is a GI-graph of $G$.
\end{theorem}

\begin{proof}
Suppose on the contrary that $\Gamma$ is not a GI-graph of $G$. By Theorem~\ref{thm2}, $\Aut(\Gamma)\neq\A_5$ or $\Sy_5$. Let $\alpha$ be a vertex of $\Gamma$, $X=\Aut(\Gamma)$ and $T$ be the socle of $X$. Then $T\neq\A_5$, $X=\hat{G}X_\alpha$ and $\hat{G}\cap T$ is a normal subgroup of $\hat{G}$. It follows that $\hat{G}\cap T=1$ or $\hat{G}$ since $\hat{G}\cong G$ is simple. If $\hat{G}\cap T=1$, then $\A_5=\hat{G}\cong\hat{G}T/T\leqslant X/T$, contrary to Schreier conjecture. Hence $\hat{G}\cap T=\hat{G}$, or equivalently, $\hat{G}\leqslant T$. Thereby we have the factorization $T=\hat{G}T_\alpha$, which is classified in Lemma~\ref{lem2}. If $T$ acts $2$-transitively on $[T{:}T_\alpha]$, then $\Gamma$ is the complete graph on $n$ vertices and $X=\Sy_n$, which implies that $\Gamma$ is a GI-graph of $G$ by Theorem~\ref{thm2}, contrary to our assumption. Consequently, $T$ does not act $2$-transitively on $[T{:}T_\alpha]$, and so we deduce from Lemma~\ref{lem2} that one of the following three cases appears:
\begin{itemize}
\item[(i)] $T=\A_6$ and $T_\alpha=\A_4$ or $\Sy_4$;
\item[(ii)] $T=\PSL_2(11)$ and $T_\alpha=\Z_{11}$;
\item[(iii)] $T=\PSL_2(29)$ and $T_\alpha=\Z_{29}{:}\Z_7$.
\end{itemize}

First suppose that case~(i) appears. As $X$ should have at least two conjugacy classes of subgroups isomorphic to $\A_5$ by Theorem~\ref{thm2}, the only possibilities for $X$ are $\A_6$ and $\Sy_6$. If $X=\A_6$, then $\A_4\leqslant X_\alpha\leqslant\Sy_4$ and hence $X$ has only one conjugacy class of vertex-transitive subgroups isomorphic to $\A_5$, which leads to a contradiction that $\Gamma$ is a GI-graph of $G$ by Theorem~\ref{thm2}. If $X=\Sy_6$, then $\A_4\leqslant X_\alpha\leqslant\Sy_4\times\Sy_2$ and hence $X$ has at most one conjugacy class of vertex-transitive subgroups isomorphic to $\A_5$, again a contradiction.

Next suppose that case~(ii) appears. As $X$ should have at least two conjugacy classes of subgroups isomorphic to $\A_5$ by Theorem~\ref{thm2}, it derives that $X=\PSL_2(11)$ and so $X_\alpha=\Z_{11}$. Since $\Gamma$ is symmetric, there exists $g\in X\setminus X_\alpha$ such that $\Gamma\cong\Cos(X,X_\alpha,X_\alpha gX_\alpha)$. Let $Y=\PGL_2(11)>X$. One can take an involution $t\in\Nor_Y(X_\alpha)$ such that $X_\alpha g^tX_\alpha=X_\alpha gX_\alpha$. Let $H=\langle X_\alpha,t\rangle=X_\alpha\langle t\rangle$, and note $t\not\in X$. Due to $X_\alpha g^tX_\alpha=X_\alpha gX_\alpha$ we have $tgt\in X_\alpha gX_\alpha$. For any $h_1,h_2\in H$, if $h_1gh_2\in X$, then either $h_1,h_2\in X_\alpha$ or $h_1,h_2\notin X_\alpha$. Further, if $h_1,h_2\notin X_\alpha$, then $h_1t,th_2\in X_\alpha$ and so $h_1gh_2=(h_1t)tgt(th_2)\in X_\alpha gX_\alpha$. This shows that $(HgH)\cap X=X_\alpha gX_\alpha$. Then the map
$$
X_\alpha x\mapsto Hx\quad\text{for }x\in X
$$
is a graph isomorphism from $\Cos(X,X_\alpha,X_\alpha gX_\alpha)$ to $\Cos(Y,H,HgH)$. However, this implies that $Y=\PGL_2(11)$ is a group of automorphisms of $\Gamma\cong\Cos(Y,H,HgH)$, contrary to the condition that $\Aut(\Gamma)=X=\PSL_2(11)$.

Finally suppose that case~(iii) appears. As $X$ should have at least two conjugacy classes of subgroups isomorphic to $\A_5$ by Theorem~\ref{thm2}, it derives that $X=\PSL_2(29)$ and so $X_\alpha=\Z_{29}{:}\Z_7$. Thus $\Gamma$ has order $|X|/|X_\alpha|=60$. Take $\beta$ to be a neighbor of $\alpha$ in $\Gamma$. Since $\Gamma$ is $X$-symmetric, $|X_\alpha|/|X_{\alpha\beta}|$ equals the valency of $\Gamma$, which is less than $60$. Hence $X_{\alpha\beta}=\Z_{29}$ or $\Z_7$. If $X_{\alpha\beta}=\Z_{29}$, then $X_{\alpha\beta}$ fixes each neighbor of $\alpha$ since $X_{\alpha\beta}\vartriangleleft X_\alpha$. This will cause a contradiction that $X_{\alpha\beta}=1$ due to the connectivity of $\Gamma$. Consequently, $X_{\alpha\beta}=\Z_7$ and $\Gamma$ is of valency $|X_\alpha|/|X_{\alpha\beta}|=29$. Note that $X$ has a maximal subgroup $K=\Z_{29}{:}\Z_{14}$ containing $X_\alpha$ such that $X$ acts $2$-transitively on $[X{:}K]$. We deduce that $X$ has an imprimitive block system $\calB=\{V_1,V_2,\dots,V_{30}\}$ on the vertex set of $\Gamma$, where $|V_1|=\dots=|V_{30}|=2$, and the quotient graph of $\Gamma$ with respect to the partition $\calB$ is complete. Moreover, denoted by $V_k$ the block in $\calB$ such that $\alpha\in V_k$, the action of $X_\alpha$ on $\calB\setminus\{V_k\}$ is transitive. Therefore, distinct neighbors of $\alpha$ lie in distinct blocks in $\calB$, and so the induced graph $\Gamma[V_i\cup V_j]$ is a perfect matching for any two blocks $V_i,V_j$ in $\calB$. Now we see that interchanging the two vertices in each $V_i$ is an automorphism of $\Gamma$. Then the kernel of $X$ acting on $\calB$ is non-trivial, contrary to the fact that $X=\PSL_2(29)$ is simple.
\end{proof}

\begin{theorem}
Let $G=\A_5$ and $\Gamma$ be a connected symmetric coset graph of $G$. If $\Aut(\Gamma)$ is vertex-primitive, then $\Gamma$ is a GI-graph of $G$.
\end{theorem}

\begin{proof}
Suppose on the contrary that $\Gamma$ is not a GI-graph of $G$. A subgroup of $G$ has order $1$, $2$, $3$, $4$, $5$, $6$, $10$ or $12$, whence the order of $\Gamma$ is $60$, $30$, $20$, $15$, $12$, $10$, $6$ or $5$. In view of Theorem~\ref{lem3} we may assume that $X:=\Aut(\Gamma)$ is not almost simple. Further, Theorem~\ref{thm2} requires $X$ to have at least two conjugacy classes of transitive subgroups isomorphic to $\A_5$. Then by~\cite[Appendix~B]{DM-book}, $X=\Hol(G)$ or $\Soc(\Hol(G))$, where the symbol $\Hol$ denotes the holomorph of a group. Let $N=\Soc(\Hol(G))=G\times G$ and $D$ be the full diagonal subgroup of $N$. Then the vertex set of $\Gamma$ can be viewed as $[N{:}D]$, with the action of $N$ by right multiplication. Moreover, let $t$ be the permutation
$$
D(g_1,g_2)\mapsto D(g_2,g_1)\quad\text{for }(g_1,g_2)\in N
$$
on $[N{:}D]$, $\alpha=D\in[N{:}D]$, $H=\langle X_\alpha,t\rangle$ and $Y=\langle X,t\rangle$. Clearly, $t$ is an involution in $Y\setminus X$. Since $\Gamma$ is symmetric, there exists $g\in X\setminus X_\alpha$ such that $\Gamma\cong\Cos(X,X_\alpha,X_\alpha gX_\alpha)$.

First suppose that $g\in\Soc(\Hol(G))$. Then $g=(g_1,g_2)$ acts on $[N{:}D]$ by right multiplication for some $g_1,g_2\in G$. Take $h_1\in G$ such that $(g_1g_2^{-1})^{h_1}=(g_1g_2^{-1})^{-1}$ and write $h_2=g_1^{-1}h_1^{-1}g_2$. For $i=1,2$ set $x_i$ to be the right multiplication of $(h_i,h_i)$ on $[N{:}D]$. It is routine to verify that $tgt=x_1gx_2\in X_\alpha gX_\alpha$. Hence $(HgH)\cap X=X_\alpha gX_\alpha$, and thus the map
$$
X_\alpha x\mapsto Hx\quad\text{for }x\in X
$$
is a graph isomorphism from $\Cos(X,X_\alpha,X_\alpha gX_\alpha)$ to $\Cos(Y,H,HgH)$. However, this implies that $Y$ is a group of automorphisms of $\Gamma\cong\Cos(Y,H,HgH)$, contrary to the condition that $\Aut(\Gamma)=X<Y$.

Next suppose that $g\in\Hol(G)\setminus\Soc(\Hol(G))$. Then there exists an involution $\tau\in X_\alpha\setminus\Soc(\Hol(G))$ such that $t\tau=\tau t$ and $g\tau^{-1}\in\Soc(\Hol(G))$. In the previous paragraph we see that $t(g\tau^{-1})t\in X_\alpha g\tau^{-1}X_\alpha$. Hence
$$
tgt=tg\tau^{-1}\tau t=t(g\tau^{-1})t\tau\in X_\alpha g\tau^{-1}X_\alpha\tau=X_\alpha gX_\alpha.
$$
Consequently, $(HgH)\cap X=X_\alpha gX_\alpha$, and so the map
$$
X_\alpha x\mapsto Hx\quad\text{for }x\in X
$$
is a graph isomorphism from $\Cos(X,X_\alpha,X_\alpha gX_\alpha)$ to $\Cos(Y,H,HgH)$. However, this implies that $Y$ is a group of automorphisms of $\Gamma\cong\Cos(Y,H,HgH)$, contrary to the condition that $\Aut(\Gamma)=X<Y$.
\end{proof}


\end{document}